\documentclass[11pt]{amsart}
\usepackage{amsmath}
\usepackage{amssymb}
\usepackage{amsfonts}
\usepackage{graphicx}
\usepackage{mathtools}
\usepackage{xcolor}
\usepackage{nameref}
\usepackage{hyperref}
\DeclareSymbolFont{AMSb}{U}{msb}{m}{n}
\DeclareMathSymbol{\I}{\mathbin}{AMSb}{"49}
\DeclareMathSymbol{\C}{\mathbin}{AMSb}{"43}

\newcommand{\N}{\mathbb{N}}

\newcommand{\F}{\mathcal{F}}

\newtheorem{theorem}{Theorem}[section]
\newtheorem{lemma}[theorem]{Lemma}

\theoremstyle{definition}
\newtheorem{definition}[theorem]{Definition}

\theoremstyle{corollary}
\newtheorem{corollary}[theorem]{Corollary}
\theoremstyle{example}
\newtheorem{example}[theorem]{Example}
\theoremstyle{note}

\theoremstyle{notation}

\numberwithin{equation}{section}
\newtheorem*{theorem*}{Theorem}

\begin{document}
\title{Recurrence in a dynamical system over adequate partial
semigroups}

\author{Md Moid Shaikh}
\address{Md Moid Shaikh, Department of Mathematics, Maharaja Manindra Chandra College, 20 Ramkanto
Bose Street, Kolkata-700003, West Bengal, India}
\email{mdmoidshaikh@gmail.com\\
mdmoidshaikh@mmccollege.ac.in}

\author{Shivam Kumar}
\address{Shivam Kumar, Department of Mathematics, Magadh University, Bodh Gaya, Aurangabad, Bihar 824101, India}
\email{shivamkumar288@gmail.com}

\author{Sourav Kanti Patra}
\address{Sourav Kanti Patra, Department of Mathematics, Kishori Sinha Mahila College, Q976+WXP, Aurangabad, Bihar 824101, India}
\email{souravkantipatra@gmail.com}



\begin {abstract}
Using tools from topological dynamics, H.~Furstenberg introduced the notion of
\emph{central sets} and established the celebrated Central Sets Theorem. The sets which satisfy the conclusion
of the Central Sets Theorem are called $C$-sets. Hindman, Maleki, and Strauss first  brought the concept of an important type of $C$-sets called the quasi-central sets which are not central sets. In 2017, A. Ghosh gave the combinatorial treatment of $C$-sets in commutative adequate partial semigroups, where $C$-sets are the sets which satisfy the conclusion of Central Sets Theorem for commutative adequate partial semigroups. In this work,  we  discuss the Quasi-central sets algebraically and dynamically for commutative adequate partial semigroups. We give dynamical characterization of members of idempotent  ultrafilters for commutative adequate partial semigroups, also we study the minimal dynamical systems for an adequate partial semigroup.

\end{abstract}

\maketitle

\section{Introduction}
Using dynamical systems (topological
dynamics), H. Furstenberg and B. Weiss gave many
fundamental results  in Ramsey Theory in \cite{frus} and  \cite{frusw}. In many cases, the use of dynamical systems to
Ramsey-theoretical problems is simpler than the algebraic or combinatorial
ones; in other cases, the dynamical notions are found to be quite interesting to be studied for themselves. In topological dynamics, one of key concepts is the concept of central sets.  In \cite[Section 6]{berg}, V. Bergelson and N. Hindman, with the help of B. Weiss, gave an algebraic characterization of central sets in $\mathbb N$. From this algebraic characterization, it follows that central sets are partition regular. Now we state the strongest combinatorial theorem about central sets which is known as the Central Sets Theorem.

Throughout this paper, we take $P_f(S)$ to be the set of all nonempty finite subsets of $S$ and $\bar{A}$ and $cl A$ to mean the same thing.
\begin{theorem}[Central Sets Theorem]
 Let $C$ be a central subset of $S$, where $(S,+)$ is a commutative semigroup. We also let $\tau$ to be the set of all sequences in
$S$.  Then there exist functions
$\alpha : \mathcal{P}_f(\tau)\to S$ and $H: \mathcal{P}_f(\tau) \to \mathcal{P}_f(\mathbb{N})$ such that
\begin{enumerate}
\item [(i)] if $F,G \in \mathcal{P}_f(\tau)$ and $F \subsetneqq G$ then $\max H(F) < \min H(G)$ and
\item [(ii)] whenever $m$ is a natural number, $G_1, G_2, \cdots, G_m \in \mathcal{P}_f(\tau)$,
$G_1 \subsetneqq G_2 \subsetneqq \cdots \subsetneqq G_m$ and for every $j \in \{1, 2, \ldots, m\}$,
$f_j \in G_j$, one has $$\sum_{j=1}^{m}\big(\alpha(G_j)+\sum_{t\in H(G_j)}f_j(t)\big)\in C.$$
\end{enumerate}
\label{t1.1}
\end{theorem}
\begin{proof}
 See \cite[Theorem 2.2]{de}.
\end{proof}
The sets which satisfy the conclusion of Central Sets Theorem are known as $C$-sets. There are some sets which are $C$-sets but not central sets. In \cite{de, hind09, HS09}, authors gave combinatorial and algebraic descriptions of $C$-sets and dynamical descriptions of  $C$-sets were found in \cite{john, jli}.

In \cite{hindcen}, an important type of $C$-sets called as quasi-central sets were first discovered and discussed algebraically as well as combinatorially. By the proof of \cite[Theorem 2.2]{de}, it was established that quasi-central sets also satisfy the conclusion of the Central Sets Theorem but the notion of $C$-sets and quasi-central sets are different which was shown in \cite{hind09}. In \cite{burns}, a dynamical characterization of quasi-central sets was obtained.

In this present work, the main object of our study is the ``partial semigroup'' and some dynamics in it, where a partial semigroup was introduced in \cite{berg94} as a pair $\left(S,\cdot\right),$ where `$\cdot$' maps a subset of $S\times S$ to $S$ and for all $a,b,c\in S$, $\left(a\cdot b\right)\cdot c=a\cdot\left(b\cdot c\right)$ in the sense that if either side is defined, then so is the other and they are equal. In  the first  two paragraphs of\cite[Section 2]{HM}, authors explained why the notion of partial semigroups are interesting objects in Ramsey Theory. Since then many researches have been done on partial semigroup, one can see \cite{HFM, G23, HP, JM, JM01} for more interesting results related to partial semigroup. In \cite{G}, A. Ghosh defined a $C$-set in an adequate partial semigroup to be a set satisfying the conclusion of the Central Sets Theorem for an adequate partial semigroup and she established a sufficient condition for being a $C$-set. Recently in \cite{pdeb25}, P. Debnath,  S. Goswami, and S. K.~Patra gave a dynamical characterization of central sets in adequate partial semigroups.

Motivated by the above, we have designed this paper in the following way: In Section 2, we discuss the preliminary results related to the Stone-\v{C}ech compactification of a discrete semigroup, topological dynamics, partial semigroup, and we complete the section 2 with a brief discussion of topological dynamics for an adequate partial semigroup. In section 3, we study the minimal dynamical systems for an adequate partial semigroup which is the analogous version of the minimal dynamical systems for a discrete semigroup  presented in \cite{hindrec}. Section 4  provides dynamical characterizations of members of idempotent  ultrafiltersa  in adequate partial semigroups. In the final section of this paper, we define quasi-central sets in an adequate commutative partial semigroup, also we give their dynamical characterizations there.

\section{Preliminary results}
Here in this section, we discuss some preliminary concepts, definitions, theorems, conventions, and results which will be
used frequently later in this work. We do not give any proofs but references to where interested readers can find proofs.

At first we discuss briefly the Stone-\v{C}ech compactification of a discrete semigroup. Take $(S,\cdot)$ to be any discrete semigroup then the Stone-\v{C}ech compactification $\beta S$ of the discrete
semigroup $(S,\cdot)$ is defined to be the set of all ultrafilters on $S$. The principal ultrafilters are identified with the
points of $S$. For a subset  $A$ of $S$, we denote $\bar{A}=\{ p\in \beta S : A\in p\}$. Then the set $\{ \bar{A} :
A\subseteq S\}$ forms a  clopen basis for a topology on $\beta S$. One can extend the operation $\cdot$ on $S$  to the
Stone-\v{C}ech compactification $\beta S$ of $S$ which makes  $(\beta S,\cdot)$ is a compact,  right topological
semigroup (meaning that for any $p\in \beta S$, the function $\rho_p:\beta S \rightarrow \beta S$ defined by
$\rho_p(q)=q\cdot p$ is continuous) with $S$ contained in its topological center (meaning that for any $x\in S$
the function $\lambda_x : \beta S \rightarrow \beta S$
defined by $\lambda_x (q)=x\cdot q$ is continuous). Let $p,q \in \beta S $, and
$A \subseteq S$, $A \in p\cdot q$ if and only if $\{ x \in S :x^{-1}A \in q\} \in p$, where
$x^{-1}A=\{ y \in S: x\cdot y \in A \}$.

\begin{definition}\cite[Definition 2.1]{john}\label{d new 2.1}

 Let $S$ be a nonempty discrete space and let
$\mathcal{K}$ be a filter on $S$.
\begin{enumerate}
 \item $\bar{\mathcal{K}}= \{p \in \beta S : \mathcal{K} \subseteq p \}$.
 \item $\mathcal{L}(\mathcal{K})=\{A\subseteq S: S\setminus A \not \in \mathcal{K}\}$.
 \end{enumerate}

\end{definition}

Theorem 3.20 of \cite{hindcen}  shows that the function $\mathcal{K} \rightarrow \bar{\mathcal{K}}$
is a bijection from the collection of all filters on $S$ onto the collection of all compact subspaces of $\beta S$.
The following  theorem relates the above two concepts nicely.

\begin{theorem}\label{t new 2.1} \cite[Theorem 2.2]{john} Let $S$ be a nonempty discrete space and $\mathcal{K}$ be a
filter on $S$. Then
\begin{enumerate}
\item  $\bar{\mathcal{K}}=\{ p \in \beta S: A \in \mathcal{L}(\mathcal{K})$ for all $A \in p$\}.
\item  Let $\mathcal{B} \subseteq \mathcal{L}(\mathcal{K})$ be closed under finite intersections. Then
there exists a $p \in \beta S$ with $\mathcal{B} \subseteq p \subseteq \mathcal{L}(\mathcal{K})$.
\end{enumerate}
\end{theorem}

\begin{proof}
Proofs follow from \cite[Theorem 3.11]{hindalg}.
\end{proof}

A nonempty subset $I$ of a semigroup $(T,\cdot)$
is said to be a {\it left ideal} of $S$ if $T\cdot I \subseteq I$, a {\it right ideal} of $S$ if
$I\cdot T \subseteq I$, and a {\it two-sided ideal} (or simply an {\it ideal}) if it is both a
left and a right ideal.  A left ideal is minimal if it does not
contain any proper left ideal. In a similar way, one can define a minimal right ideal
and the smallest ideal. One must have the smallest two-sided ideal for any compact Hausdorff right topological semigroup $(T,\cdot)$.\\
$$\begin{array}{ccc}
K(T) & = & \bigcup\{L:L \text{ is a minimal left ideal of } T\} \\
& = & \,\,\,\,\,\bigcup\{R:R \text{ is a minimal right ideal of } T\}.
\end{array}$$\\
$L\cap R$ is a group for a given a minimal left ideal $L$ and a minimal right ideal $R$ and hence it   contains
an idempotent. An idempotent belonging to the smallest ideal is minimal and conversely.\\
Now we can give the algebraic definitions of  central set and quasi-central set in a semigroup.
\begin{definition}\label{d 2.1} \begin{enumerate}
 \item Let $S$ be a discrete semigroup and let $C$ be a subset of $S$.
Then $C$ is {\it central} if and only if there is an idempotent $p$ in $K(\beta S)$ such
that $C \in p$.
\item Let $S$ be a discrete semigroup and let $C$ be a subset of $S$.
Then $C$ is said to be {\it quasi-central} if and only if there is an idempotent
$p$ in $cl K(\beta S)$ such that $C \in p$.
\end{enumerate}
\end{definition}
We now discuss some basic results of topological dynamics for an arbitrary semigroup. Let us start with the well known definition of a dynamical system.
\begin{definition}\label{d 2.4}
A {\it dynamical system} is a pair $(X,\langle T_s\rangle _{s \in S })$
such that
\begin{enumerate}

\item  $X$ is a compact Hausdorff space,

\item  $S$ is a semigroup,

\item  for each $s\in S$, $T_s:X\rightarrow X$ and $T_s$ is continuous, and

\item for all $s,t\in S$, $T_s\circ T_t=T_{st}$.
\end{enumerate}

\end{definition}
\begin{definition}\label{d new 2.5}\cite[Definition 3.1]{john}  Let $(X,\langle T_s\rangle _{s \in S})$ be a
dynamical system, $x$ and $y$ be two points in $X$, and $\mathcal{K}$ be a filter on $S$.
The pair $(x,y)$ is called jointly $\mathcal{K}$-recurrent if and only if for
every neighbourhood $U$ of $y$ we have $\{ s \in S: T_s(x) \in U$ and
$T_s(y) \in U \} \in \mathcal{L}(\mathcal{K})$.
\end{definition}

\begin{theorem}\label{t new 2.2}  Let $S$ be a semigroup, $\mathcal{K}$ be a filter on
$S$ such that $\bar{\mathcal{K}}$ is a compact subsemigroup of $\beta S$, and let
$A \subseteq S$. Then $A$ is a member of an idempotent in $\bar{\mathcal{K}}$ if
and only if there exists a dynamical system $(X,\langle T_s\rangle _{s \in S})$ with points
$x$ and $y$ in $X$ and there exists a neighbourhood $U$ of $y$ such that the
pair $(x,y)$ is jointly $\mathcal{K}$-recurrent and $A=\{ s \in S: T_s(x) \in U \}$.
\end{theorem}
\begin{proof}
See the proof of \cite[Theorem 3.3]{john}.
\end{proof}

Now we present the definitions of some important sets which are originated in topological dynamics.

\begin{definition}\label{1.3}(\cite[Definition 3.1]{hindcen}) Let $S$ be a semigroup and
let $A\subseteq S$.

(a) The set $A$ is {\it syndetic} if and only if there is some $G \in P_f(S)$ such that
    $S=\bigcup_{t \in G}t^{-1}A $.

(b) The set $A$ is {\it piecewise syndetic} if and only if there is some  $G \in P_f(S)$
such that for any $F \in P_f(S)$ there is some $x \in S$ with
$Fx\subseteq \bigcup_{t \in G}t^{-1}A $.
\end{definition}

Recall the definitions of proximality from \cite[Definition 1.2(b)]{burns}, $U(x)$ from \cite[Definition 1.5(b)]{hindrec}
and uniform recurrence in a dynamical system
 from \cite[Definition 1.2(c)]{burns}, as these definitions will be analogously discussed in the context of an adequate partial semigroup.

 \begin{definition}\label{ d new 1.4} Let $(X,\langle T_s\rangle _{s \in S})$ be a dynamical system.
\begin{enumerate}
 \item A point $y \in S$ is uniformly recurrent if and only if for every neighbourhood
    $U$ of $y$, $\{ s\in S : T_s(y) \in U  \}$ is syndetic.

\item  For $x\in X$, $U(x)=U_{X}(x)=\{p\in \beta S: T_p(x)$ is uniformly recurrent$\}$.

\item  The points $x$ and $y$ of $X$ are proximal if and only if for every neighbourhood
    $U$ of the diagonal in $X\times X$, there is some $s \in S$ such that
    $(T_s(x), T_s(y)) \in U$.
  \end{enumerate}
    \end{definition}
By \cite[Theorem 2.4]{shi}, a subset $C$ of a semigroup $S$ is central if and only if
there exist a dynamical system $(X,\langle T_s\rangle _{s \in S})$, points $x$ and $y$ of $X$
and a neighbourhood $U$ of $y$ such that $y$ is uniformly recurrent, $x$ and $y$
are proximal and $C=\{s  \in S: T_s(x) \in U  \}$.

We have presented the notions of central sets and quasi-central sets in a semigroup but in this present paper, we are interested to explore some algebraic and dynamical results related to central sets and quasi-central sets in  adequate partial semigroups. Now we give some brief descriptions of partial semigroups.

\begin{definition}[Partial semigroup]\label{d 2.2}

A {\it partial semigroup} is a pair $\left(S,\cdot\right),$ where `$\cdot$' maps a subset $X$ of $S\times S$ to $S$ and satisfies the condition that for all $a,b,c\in S$, $\left(a\cdot b\right)\cdot c=a\cdot\left(b\cdot c\right)$ in the sense that if either side is defined, then so is the other and they are equal. A partial semigroup is said to be commutative if $\left(a\cdot b\right)=\left(b\cdot a\right)$ holds for every pair $(a, b)\in X$.
\end{definition}
Note that every semigroup is a partial semigroup.
\begin{example}\label{e 2.1}
Let $\mathcal{R}$ be the collection of all matrices of order $m\times n$ with entries from  $\mathbb{Z}$, where $m,n\in \mathbb{N}$. Let us consider the usual matrix multiplication on $\mathcal{R}$. Now it is well-known  that for an $m\times n$ matrix $A$ and an $p\times q$ matrix $B$ in $\mathcal{R}$, $A\cdot B$ is defined if and only if $n=p$. Now we define `$\cdot$' in the following way:
	$$ M=
	\begin{cases}
		A\cdot B \text{ if } n=p\\
		\text{undefined otherwise}
	\end{cases}$$ then $\left(\mathcal{R},\cdot\right)$ is a partial semigroup.
\end{example}
\begin{example}\label{e 2.2}
 Let $S$ be the collection of all nonempty finite subsets of $\mathbb{N}$. Let $X=\{(A, B)\in S\times S: A\cap B=\emptyset \}$ and let $\cdot: X \longrightarrow S$ be the union. Then it is easy to show  that $\left(S,\cdot\right)$ is a commutative partial semigroup.
\end{example}
For more examples with detailed explanations, one can see \cite{HFM}.
\begin{definition}[Adequate partial semigroup]\label{d 2.3}
Let $\left(S,\cdot\right)$ be a partial semigroup.
\begin{enumerate}
 \item For $s\in S$, $\phi(s)=\{t\in S: s\cdot t \text{ is defined }\}$.
\item For $H\in \mathcal{P}_{f}\left(S\right)$,  $\sigma\left(H\right)=\cap_{s\in H}\phi \left(s\right)$.
\item The semigroup $\left(S,\cdot\right)$ is {\it adequate} if and only if $\sigma\left(H\right)\neq \emptyset$ for all $H\in \mathcal{P}_{f}\left(S\right)$.
\item $\delta S=\bigcap_{x\in S}cl_{\beta S}(\phi\left(x\right))=\bigcap_{H\in \mathcal{P}_{f}\left(S\right)}cl_{\beta S}(\sigma\left(H\right))$.
\end{enumerate}

\end{definition}

 In the example \ref{e 2.1}, $\left(\mathcal{R},\cdot\right)$ is not an adequate partial semigroup, whereas in the example \ref{e 2.2}, $\left(S,\cdot\right)$ is an adequate partial semigroup.

 Note that if $(S,\cdot)$ is a partial semigroup and  $\beta S$ is the Stone-\v{C}ech compactification of $S,$ then $(\beta S,\cdot)$ is also a partial semigroup. Clearly $\delta S\subseteq \beta S$ and , in fact, $\delta S=\beta S$ if $S$ is a semigroup.  Interestingly the set $\delta S$ becomes a semigroup under the operation `$\cdot$'. It is to be noted  that for an adequate partial semigroup $S$, $\delta S\neq \emptyset$. That means one
  can find a semigroup structure inside $\beta S$ for an adequate partial semigroup $S$, which is not possible for general partial semigroups.   We define define $x^{-1}A=\left\{y\in \phi\left(x\right):x\cdot y\in A\right\}$ for $x\in S$ and $A\subseteq S$.
 The following lemma presents the algebraic structure of  ``adequate partial semigroups''.
\begin{lemma}\label{l 2.1}
\begin{enumerate}
 \item Let $x\in S$, let $q\in \overline{\phi\left(x\right)}$, and let $A \subseteq S$. Then $$ A\in x\cdot q\iff x^{-1}A\in q,$$
 \item Let $p\in \beta S$, let $q\in \delta S$, and let $A\subseteq S$. Then $$A\in p\cdot q\iff \left\{x\in S:x^{-1}A\in q\right\}\in p.$$
\end{enumerate}

\end{lemma}
\begin{proof}
See \cite[Lemma 2.3]{HM}.
\end{proof}
\begin{theorem}\label{t 2.1}
Let $(S,\cdot)$ be an adequate partial semigroup. Then $(\delta S,\cdot)$ is a compact
Hausdorff right topological semigroup.
\end{theorem}
\begin{proof}
See  \cite[Theorem  2.10]{HM}.
\end{proof}
As a consequence of the above theorem, being $\delta S$ compact Hausdorff right topological semigroup for an adequate partial semigroup  $S$, the structure of it
is common to all objects. In particular, it has a smallest two sided ideal $K(\delta S)$. By Ellis's Theorem(Theorem 2.5 of \cite{elli}), it has an idempotent.

We end this section with very brief discussions of topological dynamics for an adequate partial semigroup.
Recently in \cite{pdeb25}, Debnath. Goswami, and Patra introduced the notion of of a dynamical system for the action of an adequate partial semigroup. The following definition is the generalization of the definition \ref{d 2.4}.
\begin{definition}\label{d 2.5}
Let $S$ be an adequate partial semigroup. A dynamical system is a pair $\left(X,\langle T_s\rangle_{s\in S}\right)$ such that
		\begin{enumerate}
			\item $X$ is a compact Hausdroff space,
			\item For each $s$, $T_{s}:X\rightarrow X$ and $T_{s}$ is continuous, and
			\item For all $s,t\in S$, $T_{s} T_{t}=T_{st}$ if $t\in \phi\left(s\right)$.
			\end{enumerate}
\end{definition}
Now we present a lemma which shows that the existence of a dynamical system for the action of an adequate partial semigroup.
\begin{lemma}\label{l 2.2}
Let $S$ be an adequate partial semigroup. Let $X=\prod_{s\in S}\left\{0,1\right\}$ and for $s\in S$ define $T_{s}:X\rightarrow X$, by
   $$T_{s}\left(x\right)\left(t\right)=
 \begin{cases}

x\left(st\right) & \text{ if } t\in\phi\left(s\right)\\
0 & \text{ if } t\notin\phi\left(s\right)
\end{cases}.$$ Then  $\left(X,\langle T_{s}\rangle_{s\in S}\right)$ is a dynamical system.
\end{lemma}
\begin{proof}
See  \cite[Lemma  2.22]{pdeb25}.
\end{proof}
Let $S$ be an adequate partial semigroup. If $S$ is a  semigroup, then define $\lambda_s: \beta S\rightarrow \beta S$ as $\lambda_s(p)=s\cdot p$ for all $p\in \beta S$. If $S$ is not a  semigroup, then fix $q\in \beta S\setminus \delta S$ and define $\lambda_s: \beta S\rightarrow \beta S$ as follows
\[
\lambda_s(p)=
\begin{cases}
s\cdot p & \text{if } p\in \delta S\\[4pt]
q & \text{if } p\in \beta S\setminus \delta S.
\end{cases}
\]

We end this section with the following theorem.
\begin{theorem}
 Let $S$ be an adequate partial semigroup. Then $\left(\beta S,\langle \lambda_s\rangle_{s\in S}\right)$ is a dynamical system.
\end{theorem}

\section{Minimal systems in an adequate partial semigroup}
In \cite{hindrec}, the authors proved that $U(x)$ is a left ideal of
$\beta S$ for any semigroup $S$. They also constructed a dynamical system for which
\[
K(\beta S)=\bigcap_{x\in X} U(x)
\quad \text{and}
\quad
clK(\beta S)=\overline{K(\beta S)}
=
\bigcap \{\, U(x) : x\in X \text{ and } U(x)\text{ is closed} \,\}.
\]
Moreover, under certain weak cancellation assumptions, they showed that
$U(x)$ properly contains $\overline{K(\beta S)}$.

In \cite{ms}, the above results were extended to dense subsemigroups of semitopological semigroups.
Motivated by these works, we establish analogous results for commutative adequate partial semigroups in this final section.

To proceed, we require several definitions and results concerning limits along ultrafilters and topological dynamics for adequate partial semigroups, which are natural analogues of the corresponding notions for arbitrary semigroups (see Chapter~4 of \cite{hindalg}). We begin with the following definition.

\begin{definition}\cite[Definition 2.8.]{pdeb25}\label{uniform and proximality}
		Let $(S,\cdot)$ be an adequate partial semigroup and let $\left(X,\langle T_{s}\rangle_{s\in S}\right)$ be a dynamical system.

		\begin{enumerate}
			\item A point $y\in S$ is {\it uniformly recurrent} if and only if for every neighborhood $U$ of $y$, $\left\{s\in S:T_{s}\left(y\right)\in U\right\} $ is syndetic.
			\item The points $x$ and $y$ of $X$ are {\it proximal} if and only if for every neighborhood $U$ of the diagonal in $X\times X$ and for each $H\in \mathcal{P}_{f}\left(S\right)$ there is some $s\in \sigma\left(H\right)$ such that $\left(T_{s}\left(x\right),T_{s}\left(y\right)\right)\in U$.
		\end{enumerate}
	\end{definition}

 Proximal points are characterized in terms of ultrafilters in the following theorem.

	\begin{lemma}\label{l 5.1}
		Let $S$ be an adequate partial semigroup. Let $\left(X,\langle T_{s}\rangle_{s\in S}\right)$ be a dynamical system and let $x,y\in X$. Then $x$ and $y$  are proximal if and only if there is some $p\in \delta S$ such that $T_{p}\left(x\right)=T_{p}\left(y\right)$.
	\end{lemma}

	\begin{proof}
	 See \cite[Lemma 2.12]{pdeb25}.
	\end{proof}
\begin{lemma}\label{l 5.2}
		Let $\left(S, \cdot\right)$ be an adequate partial semigroup.If $q\in \delta S$ and $p\in \beta S$, then $$\left(pq\right)=p\text{-}\lim_{s\in S}q\text{-}\lim_{t\in \phi\left(s\right)}{st}.$$

	\end{lemma}

\begin{proof}
	 See \cite[Lemma 2.15]{pdeb25}.
	\end{proof}
The following theorem is an analogous version of \cite[Theorem 4.5]{hindalg} for adequate partial semigroups.
\begin{theorem}\label{t 5.1}
	Let $\left(S,\cdot\right)$ be an adequate partial semigroup, let $X$ be a compact Housdroff space and $q\in \delta S$ and $p\in \beta S$. Let $\langle x_{n}\rangle_{s\in S}$  be an index family in $X$, and $p,q\in \beta S$. Then $$\left(pq\right)\text{-}\lim_{v\in S}x_{v}=p\text{-}\lim_{s\in S}q\text{-}\lim_{t\in \phi\left(s\right)}x_{st}.$$
\end{theorem}

\begin{proof}
  Define $f:S\to X$ by $f(s)=x_s$, and let $\widetilde{f}:\beta S\to X$ be the continuous extension. Then:
\[
\begin{aligned}
(pq)\text{-}\lim_{v\in S} x_v
    &= (pq)\text{-}\lim_{v\in S} \widetilde{f}(v)\\[4pt]
    &= \widetilde{f}\!\left( (pq)\text{-}\lim_{v\in S} v \right)\\[4pt]
    &= \widetilde{f}(pq)\\[4pt]
    &\stackrel{\text{Lemma \ref{l 5.2}}}{=}
       \widetilde{f}\!\left( p\text{-}\lim_{s\in S}\ q\text{-}\lim_{t\in\phi(s)} st \right)\\[4pt]
    &= p\text{-}\lim_{s\in S}\ q\text{-}\lim_{t\in\phi(s)} \widetilde{f}(st)\\[4pt]
    &= p\text{-}\lim_{s\in S}\ q\text{-}\lim_{t\in\phi(s)} x_{st}.
\end{aligned}
\]

\end{proof}
\begin{corollary}\label{action of delta s}
Let $S$ be an adequate partial semigroup. Let $\left(X,\langle T_{s}\rangle_{s\in S}\right)$ be a dynamical system. Let $p\in \beta S$ and $q\in \delta S$, then for any $x\in X$, $T_{pq}\left(x\right)=T_{p}\left(T_{q}\left(x\right)\right)$.
\end{corollary}
\begin{proof}
Let $f(s)=T_{s}(x)$. Then
\[
\begin{aligned}
T_{pq}(x)
&=(pq)\text{-}\lim_{v\in S} T_v(x)\\
&=p\text{-}\lim_{s\in S} q\text{-}\lim_{t\in\phi(s)} T_{st}(x)\\
&=p\text{-}\lim_{s\in S} q\text{-}\lim_{t\in\phi(s)} T_s\bigl(T_t(x)\bigr)\\
&=p\text{-}\lim_{s\in S} T_s\!\left(q\text{-}\lim_{t\in\phi(s)} T_t(x)\right)\\
&=p\text{-}\lim_{s\in S} T_s\!\left(T_q(x)\right)\\
&=T_p\!\left(T_q(x)\right).
\end{aligned}
\]
\end{proof}

The following lemma establishes a fundamental connection between uniformly recurrent points and minimal left ideals.

	\begin{lemma}\label{l 5.3}
		Let $S$ be an adequate partial semigroup. Let $\left(X,\langle T_{s}\rangle_{s\in S}\right)$ be a dynamical system and $L$ be a minimal left ideal of $\delta S$ and $x\in X$. The following statements are equivalent:
		\begin{enumerate}
			\item The point $x$ is a uniformly recurrent point of $\left(X,\langle T_{s}\rangle_{s\in S}\right)$.
			\item There exists $p\in L$ such that $T_{p}\left(x\right)=x$.
			\item There exists an idempotent $p\in L$ such that $T_{p}\left(x\right)=x$.
		\end{enumerate}
	\end{lemma}

	\begin{proof}
	See \cite[Lemma 2.18]{pdeb25}.
	\end{proof}

\begin{lemma}\label{l 5.4} Let $S$ be an adequate partial semigroup. Let $\left(X,\langle T_{s}\rangle_{s\in S}\right)$ be a dynamical system and $L$ be a minimal left ideal of $\delta S$ and $x\in X$. Then the following are equivalent:
 \begin{enumerate}
  \item  $x$ is uniformly recurrent.
  \item  There exists $q\in L$ such that $T_q(x)=x$.
 \item   There exists an idempotent $q\in L$ such that $T_q(x)=x$.
 \item   There exists $y\in X$  and an idempotent $q\in L$ such that $T_q(y)=x$.
\item There exists $q \in K(\delta S)$ such that $T_q(x)=x$.
\item There exists $y \in X$ and $q \in K(\delta S)$ such that $T_q(y)=x$.

\end{enumerate}
\end{lemma}
\begin{proof}

By Lemma~\ref{l 5.3}, statements (1)--(3) are equivalent.

The implication (3)$\Rightarrow$(4) is immediate.

 We now prove that (4) implies (3).
First observe that, since $q$ is idempotent and $T_q(y)=x$, we have
\[
T_q(x)
= T_q(T_q(y))
= T_{qq}(y)
= T_q(y)
= x.
\]

Let $U$ be a neighborhood of $x$ and define
\[
B=\{\, s\in S : T_s(x)\in U \,\}.
\]
We claim that $B$ is syndetic. Suppose, toward a contradiction, that $B$ is not syndetic. Then there exists $G\in \mathcal{P}_f(S)$ such that the family
\[
\mathcal{E}
=
\Big\{
\sigma(H)\setminus \bigcup_{t\in H} t^{-1}B
:
H\in \mathcal{P}_f(\sigma(G))
\Big\}
\;\cup\;
\Big\{
\sigma(K)
:
K\in \mathcal{P}_f(S)
\Big\}
\]
has the finite intersection property.

Choose $r\in \beta S$ such that $\mathcal{E}\subseteq r$.
Since $\{\sigma(K):K\in\mathcal{P}_f(S)\}\subseteq r$, it follows that $r\in \delta S$.

We claim that
\[
\delta S\, r \cap \overline{B} = \varnothing.
\]
Indeed, if not, there would exist $q\in \delta S$ such that $B\in qr$.
Then $t^{-1}B\in r$ for some $t\in \sigma(G)$, contradicting the choice of $\mathcal{E}\subseteq r$.
Thus $\delta S\, r \cap \overline{B}=\varnothing$.

Now $\delta S\, r p$ is a left ideal of $\delta S$ contained in $L$.
Since $L$ is a minimal left ideal of $\delta S$, it follows that
\[
\delta S\, r p = L.
\]
Hence we may choose $q\in \delta S\, r$ such that $qp=p$.

Then
\[
T_q(x)
= T_q(T_p(x))
= T_{qp}(x)
= T_p(x)
= x.
\]
In particular, $B \in q$. However, $q \in \delta S\, r$, which contradicts the fact that
\[
\delta S\, r \cap \overline{B} = \varnothing.
\]
This contradiction shows that $B$ is syndetic, completing the proof of (4) $\Rightarrow$ (3).

Hence, the statements (1) through (4) are equivalent.
Since (3) implies (5) and (5) implies (6), it remains to show that (6) implies (3) to establish the equivalence of all six statements.

Assume that (6) holds.
Since $L \cap \delta S\, r$ is a group, let $e$ denote the identity element of this group. Then $e q = q$, and therefore
\[
T_e(x) = T_e(T_q(y)) = T_{eq}(y) = T_q(y) = x.
\]

\end{proof}
\begin{definition}\label{d 5.2} Let $S$ be an adequate partial semigroup.
Let $(X,  \langle T_s \rangle_{s \in S})$ be a dynamical system and $x \in X$. Then
\begin{enumerate}
 \item  $U(x)=U_{X}(x)=\{p\in \delta S: T_p(x)  \text{ is uniformly recurrent}\}$.
\item  A subspace $Z$ of $X$ is called invariant  if $T_p(Z)\subseteq Z$ for every $p\in \delta S$.
\end{enumerate}
\end{definition}

\begin{corollary}\label{c 5.1} Let $S$ be an adequate partial semigroup
and $(X,  \langle T_s \rangle_{s \in S})$ be a dynamical system and $x \in X$.
\begin{enumerate}
  \item  If $x$ is uniformly recurrent  then  $\delta S=U(x)$.
 \item  For each $x \in X$, $K(\delta S)\subseteq U(x)$.
\item   For each $x \in X$, $U(x)\cap \delta S$ is a left ideal of  $\delta S$.
 \item  $(\bigcap_{x \in X}U(x))\bigcap \delta S$ is a two-sided ideal of $\delta S$.
 \end{enumerate}
\end{corollary}
\begin{proof}
\begin{enumerate}
\item Suppose $x$ is uniformly recurrent. Then by Lemma \ref{l 5.4}, $T_u(x)=x$ for some
$u\in K(\delta S)$. Thus for every $v\in \delta S$, $T_v(x)=T_v(T_u(x))=T_{vu}(x)$. Now since $uv\in K(\delta S)$,
by Lemma \ref{l 5.4} $T_v(x)$ is uniformly recurrent  and thus, $v\in U(x)$.
Therefore $\delta S= U(x)$.

\item This is immediate from  Lemma \ref{l 5.4}.

\item Let $x\in X$, $p\in U(x)\cap \delta S$ and $r\in \delta S$. By Lemma \ref{l 5.4} pick $q\in K(\delta S)$ such
that $T_q(T_p(x))=T_p(x)$. Then $T_{rp}(x)=T_r(T_q(T_p(x)))=T_{rqp}(x)$. Now $rqp\in  K(\delta S)$.
So by Lemma \ref{l 5.4}, $T_{rp}(x)$ is uniformly recurrent  and hence $rp\in U(x)\cap \delta S$.
Therefore $U(x)\cap \delta S$ is a left ideal of $\delta S$.

\item By (2) $(\bigcap_{x \in X}U(x))\bigcap \delta S$ is nonempty. So by (3) $(\bigcap_{x \in X}U(x))\bigcap \delta S$
is a left ideal of $\delta S$. So it is enough to show that $(\bigcap_{x \in X}U(x))\bigcap \delta S$  is a right ideal
of $\delta S$. To this end, let $p\in (\bigcap_{x \in X}U(x))\bigcap \delta S$ and $q\in \delta S$. Suppose $y\in X$
then $p\in U(T_q(y))$. Thus $T_{pq}(y)$ is uniformly recurrent  and so $pq\in U(y)$.
\end{enumerate}
\end{proof}

We prove the following Lemma and the next Theorem in a manner which has been shown  in \cite{hindrec}.

\begin{lemma}\label{l 5.5} Let $S$ be an adequate partial semigroup.
 Let $(X,  \langle T_s \rangle_{s \in S})$ be a dynamical system and $L$ be a minimal left ideal of $\delta S$.
 \begin{enumerate}
  \item  A subspace $Y$ of $X$ is minimal among all closed and invariant subspaces of $X$ if and only if there
 is some $x\in X$ such that $Y=\{T_p(x): p\in L\}$.

 \item  Let $Y$ be a subspace of $X$ which is minimal among all closed and invariant subspaces of $X$.
 Then every element of $Y$ is uniformly recurrent.

 \item  If $x\in X$ is uniformly recurrent  and $Y=\{T_p(x): p\in \delta S\}$, then $Y$ is minimal among all
 closed and invariant subspaces of $X$.

 \item   If $x\in X$ is uniformly recurrent  then $T_p(x)$ is uniformly recurrent  for every $p\in \delta S$.
 \end{enumerate}
 \end{lemma}
\begin{proof}
\begin{enumerate}
 \item Suppose that $Y$ is minimal among all closed and invariant subspaces of $X$. Pick $x\in Y$
 and let $Z=\{T_p(x): p\in L\}$. We show that $Z$ is a closed and invariant subspace of $Y$ and this is equal
 to $Y$. If $p\in L$ and $q\in \delta S$, then $T_q(T_p(x))=T_{qp}(x)$ and $qp\in L$. So $Z$ is invariant  and
 obviously $Z \subseteq Y$. To prove $Z$ is closed, it is enough to show that any net in $Z$ has cluster point in $Z$.\\
 To  this end, let $\langle p_{\alpha}\rangle_{\alpha\in D}$ be a net in $L$ and pick a cluster point $p$ in $L$ of
$\langle p_{\alpha}\rangle_{\alpha\in D}$.
 Then $T_p(x)$ is a cluster point of $\langle T_{p_{\alpha}}(x)\rangle_{\alpha\in D}$.

 Conversely. Let $x\in X$ and $Y=\{T_p(x): p\in L\}$. Then  $Y$ is invariant  and is closed as above.
 We now show that $Y$ is minimal among all closed invariant subspaces of $X$. Suppose that $Z$ is a subspace of
 $Y$ which is closed and invariant. We shall show that $Y \subseteq Z$. So let $y\in Y$ and pick $z\in Z$.
 Then $y=T_p(x)$ and $z=T_q(x)$ for some $p$ and $q$ in $L$. Since $Lq=L$, there exists $r\in L$ such that $rq=p$.
It follows that $T_r(z)=T_r(T_q(x))=T_{rq}(x)=T_p(x)=y$ and thus $y\in Z$ as required.

\item Let $Y$ be a subspace of $X$, which  is minimal among all closed and invariant subspaces of $X$  and $x\in Y$. Pick $y\in X$ such that
 $Y=\{T_p(x): p\in L\}$. Pick $p\in L$ such that $x=T_p(y)$. By Lemma \ref{l 5.4}, $x$ is uniformly recurrent.

 \item Let $x\in X$ be uniformly recurrent  and $Y=\{T_p(x): p\in \delta S\}$. By Lemma \ref{l 5.4}, pick $q\in L$
 such that $T_q(x)=x$. By (1), it suffices to show that $Y=\{T_p(x): p\in L\}$. To prove this, let $y\in Y$ and pick
 $p\in \delta S$ such that $y=T_p(x)$. Then $y=T_p(T_q(x))=T_{pq}(x)$ and $pq\in L$ as required.

 \item
Let $x\in X$ be uniformly recurrent  and $Y=\{T_p(x): p\in \delta S\}$. By (3) $Y$ is minimal among all
 closed and invariant subspaces of $X$  so (2) applies.

 \end{enumerate}
\end{proof}
\begin{theorem}\label{t 5.2} Let $S$ be an adequate partial semigroupand
 $x\in \delta S$. Statements (1) and (2) are equivalent and imply (3). If $e\delta S$ has a left cancelable element, all
 three are equivalent.
 \begin{enumerate}
  \item $x\in K(\delta S)$.
\item  $x\in X$ is uniformly recurrent  in the dynamical system $(\beta S,  \langle \lambda_s \rangle_{s \in S})$.

\item $\delta Sx$ is a minimal left ideal of $\delta S$.

 \end{enumerate}

\end{theorem}
\begin{proof}

 (1) implies (2). Let $x\in K(\delta S)$ and let $u$ be the identity of the group in $K(\delta S)$ to
 which $x$ belongs. Then $x=\lambda_u(x)$ so by Lemma \ref{l 5.4}, $x$ is uniformly recurrent  in the dynamical
 system $(\beta S,  \langle \lambda_s \rangle_{s \in S})$.

(2) implies (1). Let $x$ be uniformly recurrent in the dynamical system $(\beta S,  \langle \lambda_s \rangle_{s \in S})$.
By Lemma \ref{l 5.4}, there exists $q\in K(\delta S$ such that $\lambda_q(x)=x$. Then $x=qx\in K(\delta S)$.

(1) implies (3). Assume that $x\in K(\delta S)$ and pick the minimal left ideal $L$ of $\delta S$ such that $x\in L$.
Then $\delta Sx$ is a left ideal of $\delta S$ contained in $L$. So $L=\delta Sx$. Now assume that $\delta S$ has a
left cancelable element $z$ and $\delta Sx$ is a minimal left ideal of $\delta S$. Pick an idempotent $u\in \delta Sx$.
Then $zx\in \delta S$. So by \cite[Lemma 1.30]{hindalg}, $zx=zxu$ and therefore
$x=xu\in \delta Sx\subseteq K(\delta S)$.
\end{proof}

\begin{corollary}\label{c 5.2} Let  $S$ be an infinite adequate partial semigroup  and $x\in K(\delta S)$.
Then $\delta S\subseteq U(x)$ with respect to the dynamical system $(\beta S,  \langle \lambda_s \rangle_{s \in S})$.
\end{corollary}

\begin{proof}
By Theorem \ref{t 5.2}, $x$ is uniformly recurrent, so by Lemma \ref{l 5.5},
$\delta S\subseteq U(x)$.
\end{proof}

\begin{corollary}\label{c 5.3} Let $S$ be an adequate partial semigroupand  and $p, q\in e\delta S$. Statements
(1) and (2) are equivalent and imply (3). If $\delta S$ has a left cancelable element, all three statements are equivalent
 \begin{enumerate}
  \item  $q+p\in K(\delta S)$.
 \item  $q\in U(p)$ with respect to the dynamical system $(\beta S,  \langle \lambda_s \rangle_{s \in S})$.
 \item  $\delta Sqp$ is a minimal left ideal of $\delta S$.
 \end{enumerate}
 \end{corollary}

\begin{proof}
We have that $q\in U(p)$ if and only if $\lambda_q(p)$ is uniformly recurrent  and
 $\lambda_q(p)=qp$ so Theorem \ref{t 5.2} applies.
\end{proof}

\begin{corollary}\label{c 5.4}  Let $S$ be an adequate partial semigroupand. The following are equivalent.
\begin{enumerate}
 \item There exists $p\in \delta S\setminus K\delta S)$ such that $K(\delta S)\subset U(p)$ with respect to the dynamical
system $(\beta S,  \langle \lambda_s \rangle_{s \in S})$.
\item $K(\delta S)$ is not prime.
\end{enumerate}
\end{corollary}
\begin{proof}
The proof is an immediate consequence of Corollary \ref{c 5.3}.
\end{proof}

\section{Dynamical characterization of members of idempotent  ultrafilters}
\begin{definition}\label{d 3.7} Let $S$ be a commutative adequate partial semigroup,
and $(X,\langle T_s\rangle _{s \in S})$ be a
dynamical system, $x$ and $y$ be two points in $X$, and $\mathcal{K}$ be a filter on $S$.
The pair $(x,y)$ is called jointly $\mathcal{K}$-recurrent if and only if for
every neighbourhood $U$ of $y$ we have $\{ s \in S: T_s(x) \in U$ and
$T_s(y) \in U \} \in \mathcal{L}(\mathcal{K})=\{A\subseteq S: S\setminus A \not \in \mathcal{K}\}$.
\end{definition}
\begin{lemma}\label{lem:technical}
Let $(X, \langle T_s \rangle_{s \in S})$ be a dynamical system, let $x, y \in X$, and let $\mathcal{K}$ be a filter on $S$ such that $\overline{\mathcal{K}}$ is a compact subsemigroup of $\beta S$.
The following statements are equivalent:
\begin{enumerate}
    \item The pair $(x, y)$ is jointly $\mathcal{K}$-recurrent.
    \item There exists $p \in \overline{\mathcal{K}}$ such that
    \[
        T_p(x) = y = T_p(y).
    \]
    \item There exists an idempotent $p \in \overline{\mathcal{K}}$ such that
    \[
        T_p(x) = y = T_p(y).
    \]
\end{enumerate}
\end{lemma}

\begin{proof}
\textbf{(a) $\implies$ (b).}
For each neighborhood $U$ of $y$, define
\[
B_U = \{ s \in S : T_s(x) \in U \text{ and } T_s(y) \in U \}.
\]
Since $B_{U \cap V} = B_U \cap B_V$ for any neighborhoods $U, V$ of $y$, the collection
\[
\mathcal{B} = \{ B_U : U \text{ is a neighborhood of } y \}
\]
is closed under finite intersections.

By assumption, $\mathcal{B} \subseteq \mathcal{L}(\mathcal{K})$, the filter generated by $\mathcal{K}$.
Hence, by the standard ultrafilter lemma, there exists $p \in \overline{\mathcal{K}}$ such that $\mathcal{B} \subseteq p$.

Then for every neighborhood $U$ of $y$:
\[
\{ s \in S : T_s(x) \in U \} \in p \quad \text{and} \quad \{ s \in S : T_s(y) \in U \} \in p,
\]
which implies $T_p(x) = y = T_p(y)$.

\medskip
\textbf{(b) $\implies$ (c).}
Let
\[
M = \{ p \in \overline{\mathcal{K}} : T_p(x) = y = T_p(y) \}.
\]
It suffices to show that $M$ is a compact subsemigroup of $\beta S$.

\textit{Compactness:} $M$ is nonempty by assumption.
To show $M$ is closed, let $p \not\in M$. Then either $T_p(x) \neq y$ or $T_p(y) \neq y$.

If $T_p(x) \neq y$, pick a neighborhood $U$ of $y$ such that
\[
A := \{ s \in S : T_s(x) \in U \} \notin p.
\]
Then $S \setminus A \in p$, and $\overline{S \setminus A}$ is a basic neighborhood of $p$ in $\beta S$ that misses $M$.
The case $T_p(y) \neq y$ is similar. Hence $M$ is closed, and therefore compact.

\textit{Subsemigroup:} Let $q, r \in M$. Then
\[
T_{qr}(x) = T_q(T_r(x)) = T_q(y) = y \quad \text{and} \quad T_{qr}(y) = T_q(T_r(y)) = T_q(y) = y.
\]
Thus $qr \in M$, so $M$ is a subsemigroup.

\medskip
By Ellis's lemma (every compact subsemigroup of $\beta S$ contains an idempotent), $M$ contains an idempotent $p$, which proves (c).
\end{proof}

\begin{lemma}\label{shift of sequence}
Let $S$ be an adequate partial semigroup.
Let $X=\prod_{s\in S}\{0,1\}$ and for each $s\in S$ define $T_{s}\colon X\to X$ by
\[
T_{s}(x)(t)=
\begin{cases}
x(ts), & \text{if } s\in \phi(t),\\[4pt]
0, & \text{if } s\notin \phi(t).
\end{cases}
\]
Then $\left(X,\langle T_{s}\rangle_{s\in S}\right)$ is a dynamical system.
\end{lemma}

\begin{proof}
Let $s\in S$.
To see that $T_{s}$ is continuous, it is enough to show that $\pi_{t}T_{s}$ is continuous for each $t\in S$.
For $x\in X$,
\[
(\pi_{t}T_{s})(x)=\pi_{t}(T_{s}x)=T_{s}(x)(t),
\]
and
\[
T_{s}(x)(t)=
\begin{cases}
x(ts)=\pi_{ts}(x), & \text{if } s\in\phi(t),\\[4pt]
0, & \text{if } s\notin\phi(t).
\end{cases}
\]
Hence $T_{s}$ is continuous because
\[
\pi_{t}T_{s}=
\begin{cases}
\pi_{ts}, & \text{if } s\in\phi(t),\\[4pt]
\text{constant}, & \text{if } s\notin\phi(t).
\end{cases}
\]

\medskip
Now let $s,t\in S$.
For $x\in X$,
\[
T_{t}(x)(u)=
\begin{cases}
x(ut), & \text{if } t\in \phi(u),\\[4pt]
0, & \text{if } t\notin \phi(u).
\end{cases}
\]
Thus,
\[
T_{s}(T_{t}x)(u)=
\begin{cases}
T_{t}(x)(us), & \text{if } s\in\phi(u),\\[4pt]
0, & s\notin\phi(u).
\end{cases}
\]

So
\[
T_{s}T_{t}(x)(u)=
\begin{cases}
x(ust), & \text{if } s\in\phi(u) \text{ and } t\in\phi(us),\\[4pt]
0, & \text{if } s\notin\phi(u) \text{ or } t\notin\phi(us) .
\end{cases}
\]
Therefore $T_{s}T_{t}=T_{st}$ whenever $t\in \phi(s)$, completing the proof.
\end{proof}

\begin{theorem}\label{thm:main-result}
Let $(S, \cdot)$ be an adequate partial  semigroup, let $\mathcal{K}$ be a filter on $S$ such that $\overline{\mathcal{K}}$ is a compact subsemigroup of $\delta S$, and let $A \subseteq S$.
Then $A$ is a member of an idempotent in $\overline{\mathcal{K}}$ if and only if there exists a dynamical system $(X, \langle T_s \rangle_{s \in S})$, points $x, y \in X$, and a neighborhood $U$ of $y$ such that:
\begin{enumerate}
    \item the pair $(x, y)$ is jointly $\mathcal{K}$-recurrent, and
    \item $A = \{ s \in S : T_s(x) \in U \}$.
\end{enumerate}
\end{theorem}

\begin{proof}
\textbf{($\Rightarrow$)}
Let $R = S \cup \{ e \}$ be the  adequate partial semigroup obtained by adjoining an identity $e$ to $S$ (even if $S$ already has an identity).
Give $\{0,1\}$ the discrete topology and $X = \{0,1\}^R$ the product topology. Then $X$ is a compact Hausdorff space.

For each $s \in S$, define $T_s \colon X \to X$ by
\[
T_s(f) = f \circ \rho_s,
\]
where $\rho_s \colon R \to R$ is right multiplication by $s$.
Then $(X, \langle T_s \rangle_{s \in S})$ is a dynamical system (see \cite[Theorem 19.14]{hindalg}).

Let $x = \mathbf{1}_A$ be the characteristic function of $A$, pick an idempotent $r \in \overline{\mathcal{K}}$ with $A \in r$, and put $y = T_r(x)$.
Then
\[
T_r(y) = T_r(T_r(x)) = T_{rr}(x) = T_r(x) = y,
\]
so by , the pair $(x, y)$ is jointly $\mathcal{K}$-recurrent.

Define the (subbasic) neighborhood
\[
U = \{ w \in X : w(e) = y(e) \}.
\]
Observe that $y(e) = 1$: since $y = T_r(x)$, pick $s \in A$ with $T_s(x) \in U$. Then
\[
y(e) = T_s(x)(e) = x(\rho_s(e)) = x(es) = x(s) = \mathbf{1}_A(s) = 1.
\]

Finally, for each $s \in S$:
\[
\begin{aligned}
s \in A &\iff \mathbf{1}_A(s) = 1 \\
&\iff x(s) = 1 \\
&\iff x(es) = 1 \\
&\iff (x \circ \rho_s)(e) = 1 \\
&\iff T_s(x)(e) = y(e) \\
&\iff T_s(x) \in U.
\end{aligned}
\]
Hence, $A = \{ s \in S : T_s(x) \in U \}$.

\medskip
\textbf{($\Leftarrow$)}
Suppose we have a dynamical system $(X, \langle T_s \rangle_{s \in S})$, points $x, y \in X$, and a neighborhood $U$ of $y$ such that $(x, y)$ is jointly $\mathcal{K}$-recurrent and $A = \{ s \in S : T_s(x) \in U \}$.

By , there exists an idempotent $r \in \overline{\mathcal{K}}$ such that $T_r(x) = y = T_r(y)$.
Since $U$ is a neighborhood of $y$, it follows that
\[
A = \{ s \in S : T_s(x) \in U \} \in r.
\]
Thus, $A$ is a member of an idempotent in $\overline{\mathcal{K}}$.
\end{proof}

\section{Quasi-central sets in adequate partial semigroups and their dynamical characterization}
The sets which satisfy the conclusion
of the Central Sets Theorem are called $C$-sets. In \cite{hindcen}, Hindman, Maleki, and Strauss first observed that there are some $C$-sets which are not central sets. Hence they have brought the concept of an important type of $C$-sets called the quasi-central sets. The definition of  quasi-central sets in an arbitrary semigroup was first introduced algebraically in \cite[Definition 1.2]{hindcen} and combinatorially characterized in \cite[Theorem 3.7]{hindcen}. The quasi-central sets also satisfy the conclusion of the Central Sets Theorem which follows from \cite[Theorem 2.2]{de}. In  \cite{burns}, quasi-central sets and their dynamical characterizations were extensively studied. Here in this section, we define quasi-central sets in an adequate  partial semigroup, also we give their dynamical characterizations. Now we describe some basic results about syndetic sets and piecewise syndetic sets in an arbitrary partial semigroup.

\begin{definition}{\cite[Definition 3.3]{JM}}\label{syndetic in partial}
Let $\left(S,\cdot\right)$ be an adequate partial semigroup and let $A\subseteq S$.
\begin{enumerate}
 \item The set $A$ is $\check{c}$-{\it syndetic} if and only if there is some $H\in \mathcal{P}_{f}\left(S\right)$ such that $\sigma\left(H\right)\subseteq\bigcup_{t\in H}t^{-1}A$.
 \item Then $A$  is  {\it syndetic} if  for every left ideal $L$ of $\delta S$, $\overline{A}\cap L\neq\emptyset$.
\end{enumerate}
	\end{definition}
	One can see \cite[Theorem 3.4]{JM} for checking that the notions “syndetic” and and “$\check{c}$-syndetic” are not equivalent, although \cite[Theorem 3.6]{JM} tells that every syndetic set is a $\check{c}$-syndetic set in an adequate partial semigroup. In \cite{pdeb25}, authors introduced a new variation of syndetic set.
\begin{definition}{\cite[Definition 2.6]{pdeb25}}\label{new syndetic set}
Let $S$ be an adequate partial semigroup and let $A\subseteq S$. Then $A$  is  $\check{a}$-syndetic if  and only if for every left ideal $L$ of $\delta S$, $\overline{A}\cap L\neq\emptyset$.
\end{definition}

The following theorem characterizes $\check{a}$-syndetic set combinatorially.

	\begin{theorem}\label{t 4.1}
		Let $S$ be an adequate partial semigroup and let $A\subseteq S$. Then $A$  is  $\check{a}$-syndetic if  and only if for every left ideal $L$ of $\delta S$, $\overline{A}\cap L\neq\emptyset$.
	\end{theorem}
\begin{proof}
 See \cite[Theorem 2.7]{pdeb25}.
\end{proof}
Now we focus on piecewise syndetic sets. Let us start with the following definition.
\begin{definition}{\cite[Definition 3.1]{HP}}\label{piecewise syndetic in partial}
Let $\left(S,\cdot\right)$ be an adequate partial semigroup and let $A\subseteq S$, $\F$ be the collection of all adequate sequences in $S$, and $\mathcal J_m=\{t\in \N^m: t(1)<t(2)<\ldots<t(m))\}$.
\begin{enumerate}
 \item The set $A$ is {\it piecewise syndetic} if and only if $\overline{A}\cap K(\delta S)\neq\emptyset$.

 \item The set $A$ is a $J$-{\it set} if and only if for all   $F\in \mathcal{P}_{f}\left(\F\right)$ and all   $L\in \mathcal{P}_{f}\left(S\right)$,  there exist $m\in \N$, $a\in S^{m+1}$, $t\in \mathcal J_m$ such that for all $f\in F$, $(\prod_{i=1}^m a(i)\cdot f(t(i)))\cdot a(m+1)\in A\cap \sigma(L)$.

 \item $J(S)=\{p\in \delta S: (\forall A\in p)(A \text{ is a } J\text{-set})\}$.
\end{enumerate}
	\end{definition}

As was the case with “syndetic”, one can see \cite[Theorem 4.6]{JM} for checking that the notions “piecewise syndetic” and and “$\check{c}$- piecewise syndetic” are not equivalent, although \cite[Theorem 4.5]{JM} tells that every piecewise syndetic set is a $\check{c}$-piecewise syndetic set in an adequate partial semigroup.

Now we introduce the definition of quasi-central sets in a commutative adequate partial semigroup.
\begin{definition}\label{d 4.4}
 Let $\left(S,\cdot\right)$ be an adequate partial semigroup and let $A\subseteq S$. Then $A$ is said to be a quasi-central set if and only if there is an idempotent $p$  in  $cl K(\delta S)$ such that $A\in p$.
\end{definition}
Now we show that quasi-central sets satisfy  the Central Sets Theorem in a commutative adequate partial semigroup.

\begin{theorem}\label{t 4.2}
Let $\left(S,\cdot\right)$ be an  adequate partial semigroup. Then $J(S)$  is a compact two-sided ideal of $\delta S$.
\end{theorem}

\begin{proof}
 See \cite[Corollary 3.4]{HP}.
\end{proof}

\begin{theorem}\label{t 4.3}
Let $(S, \cdot)$ be an adequate partial semigroup and  $A$  be  a quasi-central subset of $S$. There exist functions
$m:\ P_f(\F) \rightarrow \N$, $\alpha \in \bigtimes_{F \in \ P_f(\F)} S^{m(F)+1}$, and $\tau \in \bigtimes_{F \in \ P_f(\F)} \mathcal J_{m(F)}$
such that
\begin{enumerate}
\item  if $F, G \in \ P_f(\F)$ and $F \subsetneq G$, then $\tau(F)(m(F)) < \tau(G)(1)$, and
\item  if $n \in \mathbb{N}$, $G_1,G_2, \ldots , G_n \in \ P_f(\F)$, $G_1 \subsetneq G_2\subsetneq
\cdots \subsetneq G_n$, and for each $i \in \{1,2, \ldots ,n\}$, $f_i \in G_i$, then
\end{enumerate}
one has $$\prod_{i=1}^{n}((\prod_{j=1}^{m(G_i)}\alpha(G_i)(j)\cdot f_i(\tau(G_i)(j)))\cdot \alpha(G_i)(m(G_i)+1))\in A.$$

\end{theorem}
\begin{proof}

By Definition \ref{d 4.4}, choose an idempotent $p$  in  $cl K(\delta S)$ such that $A\in p$. Also Theorem \ref{t 4.2} asserts that $cl K(\delta S)\subseteq J(S)$. So, $p$ is an idempotent in $J(S)$ such that $A\in p$. Now use {\cite[Theorem 3.6]{HP}} to conclude the result.
\end{proof}

Now we shall focus on dynamical characterization of quasi-central central sets in a commutative adequate partial semigroup.
Now we introduce the notion of jointly intermittently uniform recurrence
in a commutative adequate partial semigrou just like in an arbitrary semigroup.
\begin{definition}\label{4.9} Let $(S, \cdot)$ be a commutative adequate partial semigroup. Let $(X, \langle T_s\rangle_{s\in S})$ be a dynamical system
and let $x,y \in X$. The pair $(x,y)$ is {\it jointly intermittently uniformly recurrent}
 if and only if for every neighbourhood $U$ of $y$,
the set $\{s \in S : T_s(x) \in U \text{ and } T_s(y)\in U\}$ is  piecewise syndetic.
\end{definition}
\begin{lemma}\label{l 4.10}   Let $(S, \cdot)$ be a commutative adequate partial semigroup and
let $\mathcal{K} = \{A \subseteq S : S \setminus A \text{ is not  piecewise syndetic subset of } S\}$.
Then $\mathcal{K}$ is a filter on $S$ with $cl K(\delta S) = \overline{\mathcal{K}}$,
which is a compact subsemigroup of $\delta S$.
\end{lemma}
\begin{proof}
By the construction of $\mathcal{K}$ and Definition \ref{piecewise syndetic in partial}(2), we have
$\mathcal{K} = \bigcap K(\delta S)$. Using Theorem 3.20(b) of \cite{hindalg}, we have
$\mathcal{K}$ is a filter and $\overline{\mathcal{K}} = clK(\delta S)$. By \cite[Theorem 2.15]{hindalg}, $clK(\delta S)$
is a right ideal of $\delta S$, so in particular, $\overline{\mathcal{K}}$ is a compact subsemigroup of $\delta S$.
Therefore $clK(\delta S)) = \overline{\mathcal{K}}$ is a compact subsemigroup of $\beta S$.
\end{proof}

Finally, we shall end this with the dynamical characterization of quasi-central sets in adequate partial semigroups.
\begin{theorem}\label{4.11}  Let $(S, \cdot)$ be a commutative adequate partial semigroup and let $A \subseteq S$.
The set $A$ is quasi-central near $e$ if and only if there exists a dynamical system $(X,\langle T_s \rangle_{s \in S})$,
points $x$ and $y$ in $X$, and a neighbourhood $U$ of $y$ such that the
pair $(x,y)$ is jointly intermittently uniformly recurrent and $A = \{s \in S : T_s(x) \in U\}$.
\end{theorem}
\begin{proof}
 To prove this theorem, we use Theorem \ref{thm:main-result}.
 \\ Let $\mathcal{K} = \{B \subseteq S : S\setminus B \text{ is not a  piecewise syndetic setm of }S\}$.
Clearly $\mathcal{L}(\mathcal{K}) = \{A \subseteq S : A \text{ is a piecewise syndetic of }S\}$.
By Lemma  \ref{l 4.10}, we have $\mathcal{K}$ is a filter and $\overline{\mathcal{K}} = cl K(\delta S)$ which is a
compact subsemigroup of $\delta S$. Now we can apply Theorem \ref{thm:main-result} to prove our required statement.
\end{proof}

\end{document}